\documentclass[11pt,a4paper]{article}

\usepackage{epsf,epsfig,amsfonts,amsgen,amsmath,amstext,amsbsy,amsopn,amsthm,amssymb
}
\usepackage{ebezier,eepic}
\usepackage{color}
\usepackage{multirow}
\usepackage{url}
\usepackage{soul,cite}

\newtheorem{theorem}{Theorem}

\newtheorem{prop}{Proposition}
\newtheorem{lemma}{Lemma}
\newtheorem{false statement}{False statement}

\theoremstyle{definition}

\newtheorem{claim}{Claim}

\newtheorem{conjecture}{Conjecture}
\newtheorem{corollary}[claim]{Corollary}

\renewcommand{\theenumi}{\rm (\roman{enumi})}

\newcommand {\relabel}[1] {\label{#1} \red{[*: #1]}}
\newcommand {\rebibitem}[1] {\bibitem{#1} \red{[*: #1]}} 
\def\relabel {\label} \def\rebibitem {\bibitem}  

\begin{document}

\title
{Proving a conjecture on chromatic polynomials
by counting the number of acyclic orientations\thanks{This article is partially supported by
NTU AcRf Project (RP 3/16 DFM) of Singapore and
NSFC grants (No. 11701401, 11961070 and 11971346).}}
\date{}

\def \bg {\hspace{0.3 cm}}

\author{Fengming Dong\thanks{Corresponding author. Email: fengming.dong@nie.edu.sg and donggraph@163.com. },\bg Jun Ge,\bg Helin Gong
\\
Bo Ning,\bg Zhangdong Ouyang\bg and\bg Eng Guan Tay}

\maketitle

\begin{abstract}
The chromatic polynomial $P(G,x)$ of a graph $G$ of order $n$ can be expressed as $\sum\limits_{i=1}^n(-1)^{n-i}a_{i}x^i$, where
$a_i$ is interpreted as the number of broken-cycle free spanning subgraphs of $G$ with exactly $i$ components. The parameter
$\epsilon(G)=\sum\limits_{i=1}^n (n-i)a_i/\sum\limits_{i=1}^n a_i$
is the mean size of a broken-cycle-free spanning subgraph of $G$.
In this article, we confirm and strengthen a conjecture proposed by Lundow and Markstr\"{o}m in 2006
that $\epsilon(T_n)< \epsilon(G)<\epsilon(K_n)$ holds for any
connected graph $G$ of order $n$ which is neither the complete graph $K_n$ nor a tree $T_n$ of order $n$.
The most crucial step of our proof is to obtain the
interpretation of all $a_i$'s by the number of acyclic
orientations of $G$. 
\end{abstract}

\medskip



\section{Introduction}

All graphs considered in this paper are simple graphs.
For any graph $G=(V, E)$ and any positive integer $k$,
a {\it proper $k$-coloring} $f$ of $G$ is a mapping
$f: V\rightarrow \{1, 2, \ldots, k\}$
such that $f(u)\neq f(v)$ holds
whenever $uv\in E$.
The chromatic polynomial of $G$
is the function $P(G, x)$ such that $P(G, k)$
counts the number of proper $k$-colorings of
$G$ for any positive integer $k$.
In this article, the variable $x$ in $P(G,x)$
is a real number. 
The study of chromatic polynomials is one of the
most active areas in graph theory. 
For basic concepts and properties on chromatic polynomials,
we refer the reader to the monograph~\cite{DKT2005}.
For the most celebrated results on this topic, 
we recommend surveys~\cite{Dong2020,Jackson2015,Royle2009, RT1988}.

The first interpretation of the coefficients
of $P(G,x)$ was provided by Whitney~\cite{Whitney1932}:
for any simple graph $G$ of order $n$ and size $m$,
\begin{align}\relabel{int1}
P(G,x)=\sum_{i=1}^{n}\left( \sum_{r=0}^m (-1)^r N(i,r) \right )x^i,
\end{align}
where $N(i,r)$ is the number of spanning subgraphs of $G$
with exactly $i$ components and $r$ edges.
Whitney further simplified (\ref{int1})
by introducing the notion of broken cycles.
Let $\eta:E\rightarrow \{1,2,\ldots,|E|\}$
be a bijection.
For any cycle $C$ in $G$, the path $C-e$ is called a
{\it broken cycle} of $G$ with respect to
$\eta$, where $e$ is the edge on $C$
with $\eta(e)\le \eta(e')$ for
every edge $e'$ on $C$.
When there is no confusion, a broken cycle of $G$
is always assumed to be with respect to
a bijection $\eta:E\rightarrow \{1,2,\ldots,|E|\}$.

\begin{theorem}[\cite{Whitney1932}]\relabel{brokencycle}
Let $G=(V,E)$ be a graph of order $n$
and $\eta:E\rightarrow \{1,2,\ldots,|E|\}$ be a bijection.
Then,
\begin{align}\relabel{int2}
P(G,x)=\sum_{i=1}^{n} (-1)^{n-i}a_i(G)x^i,
\end{align}
where $a_i(G)$ is the number of
spanning subgraphs of $G$
with $n-i$ edges and $i$ components
which do not contain broken cycles.
\end{theorem}

Let $G$ be a simple graph of order $n$.
When there is no confusion,
$a_i(G)$ is written as $a_i$ for short.
Clearly, by Theorem~\ref{brokencycle},
$P(G,x)$ is indeed a polynomial in $x$
in which the constant term is $0$,
the leading coefficient $a_n$ is $1$
and all coefficients  are integers
alternating in signs.
Thus, $(-1)^nP(G,x)>0$ holds for all $x<0$.

The concept of broken cycles has the following connection
with Tutte's work of expressing the Tutte polynomial
${\bf T}_G(x,y)$ of a connected graph $G$ 
in terms of spanning trees \cite{Crapo1969, Tutte1954}:
\begin{align}\label{Tuute-ex1}
{\bf T}_G(x,y)=\sum_{T}x^{ia_{\omega}(T)}y^{ea_{\omega}(T)},
\end{align}
where the sum runs over all spanning trees of $G$ 
and $ia_{\omega}(T)$ and $ea_{\omega}(T)$
are respectively the internal and external activities
of $T$ with respect to a bijection
$\omega: E\rightarrow \{1,2,\ldots,|E|\}$.
If we take $\omega$ to be $\eta$, then
$ea_{\eta}(T)$ is exactly the number of edges $e\in E(G)\setminus E(T)$
such that $\eta(e)\le \eta(e')$ holds for all edges $e'$
on the unique cycle $C$ of $T\cup e$.
As $G$ is a simple graph, $ea_{\eta}(T)$ equals the
number of broken cycles contained in $T$ with respect to $\eta$.
In particular, $ea_{\eta}(T)=0$ if and only if
$T$ does not contain broken cycles with respect to $\eta$.
By Theorem~\ref{brokencycle}, $a_1(G)$ is
the number of spanning trees $T$ of $G$ with $ea_{\eta}(T)=0$.
If
\begin{align}\label{Tuute-ex2}
{\bf T}_G(x,y)=\sum\limits_{i\ge 0, j\ge 0} c_{i,j}x^iy^j,
\end{align}
then $a_1(G)=\sum_{i\ge 0}c_{i,0}={\bf T}_G(1,0)$.

As in \cite{LM2006},
for $i=0,1,2,\ldots,n-1$,
we define $b_i(G)$ (or simply $b_i$)
as the probability
that a randomly chosen broken-cycle-free spanning subgraph
of $G$  has size $i$.
Then
\begin{align}\relabel{prob-bi}
b_i=\frac{a_{n-i}}{a_1+a_2+\cdots+a_n},
\quad \forall i=0, 1, \ldots, n-1.
\end{align}
Let $\epsilon(G)$ denote the mean size of a
broken-cycle-free spanning subgraph of $G$. Then
\begin{align}\relabel{meansize}
\epsilon(G)=
\sum_{i=0}^{n-1} ib_i
=\frac{(n-1)a_1+(n-2)a_2+\cdots+a_{n-1}}{a_1+a_2+\cdots+a_n}.
\end{align}

An elementary property of $\epsilon(G)$
is given below.

\begin{prop}[\cite{LM2006}]\relabel{prop-eps}
For any graph $G$ of order $n$,
$\epsilon(G)=n+\frac{P'(G, -1)}{P(G, -1)}$.
\end{prop}

Let $T_n$ denote a tree of order $n$
and $K_n$ denote the complete graph of order $n$.
By Proposition~\ref{prop-eps},
$\epsilon(T_n)=\frac{n-1}{2}$, ,
while
\begin{align}\label{constant}
\epsilon(K_n)=n-\left(1+\frac{1}{2}+\cdots+\frac{1}{n}\right)\sim n-\log n-\gamma
\end{align}
as $n\rightarrow \infty$, where $\gamma\approx 0.577216$
is the Euler-Mascheroni constant.

Lundow and Markstr\"{o}m~\cite{LM2006}
proposed the following conjecture on $\epsilon(G)$.

\begin{conjecture}
[\cite{LM2006}]\relabel{mainconj}
For any connected graph $G$ of order $n$,
where $n\ge 4$,
if $G$ is neither $K_n$ nor a $T_n$,
then $\epsilon(T_n)<\epsilon(G)<\epsilon(K_n)$.
\end{conjecture}

In this paper, we aim to prove and strengthen
Conjecture~\ref{mainconj}.
For any graph $G$, define the function $\epsilon(G,x)$
as follows:
\begin{align}\relabel{epsi-G}
\epsilon(G,x)=\frac{P'(G, x)}{P(G, x)}.
\end{align}
By Proposition~\ref{prop-eps},
$\epsilon(G)=n+\epsilon(G,-1)$ holds for every graph
$G$ of order $n$.
Thus, for any graphs $G$ and $H$ of the same order, $\epsilon(G)<\epsilon(H)$
if and only if $\epsilon(G,-1)<\epsilon(H,-1)$.
Conjecture~\ref{mainconj} is equivalent
to the statement that
$\epsilon(T_n,-1)<\epsilon(G,-1)<\epsilon(K_n,-1)$
holds for any connected graph $G$ of order $n$
which is
neither $K_n$ nor a  $T_n$.

A graph $Q$ is said to be {\it chordal}
if $Q[V(C)]\not\cong C$
for every cycle $C$ of $Q$ with $|V(C)|\ge 4$,
where $Q[V']$ is the subgraph of $Q$ induced by $V'$
for $V'\subseteq V(G)$.
In Section~\ref{firstIn},
we will establish the following result.

\begin{theorem}\relabel{compare-Q}
For any graph $G$,
if $Q$ is a chordal and proper spanning subgraph
of $G$, then
$\epsilon(G,x)>\epsilon(Q,x)$ holds for all  $x<0$.
\end{theorem}

Note that any tree is a chordal graph
and any connected graph contains a spanning tree.
Thus, we have the following corollary
which obviously implies the first part of
Conjecture~\ref{mainconj}.

\begin{corollary}\relabel{compare-T}
For any connected graph $G$ of order $n$
which is not a tree,
$\epsilon(G, x)>\epsilon(T_n, x)$
holds for all  $x<0$.
\end{corollary}

The second part of Conjecture~\ref{mainconj}
is extended to the inequality
$\epsilon(K_n,x)>\epsilon(G,x)$
for any non-complete graph $G$ of order $n$ and all $x<0$.
In order to prove this inequality,
we will show in Section~\ref{SecondIn} that
it suffices to establish the following result.

\begin{theorem}\relabel{average-th}
For any non-complete graph $G=(V,E)$ of order $n$,
\begin{align}
(-1)^{n}(x-n+1)\sum_{u\in V}P(G-u, x)
+(-1)^{n+1}nP(G, x)> 0
\relabel{right-2}
\end{align}
holds for all $x<0$.
\end{theorem}

Note that the left-hand side of
(\ref{right-2}) vanishes when $G\cong K_n$.
Theorem~\ref{average-th} will be proved in
Section~\ref{finalproof},
based on Greene $\&$ Zaslavsky's interpretation in \cite{GZ1983}
for coefficients $a_i(G)$'s
of $P(G,x)$ by acyclic orientations
introduced in Section~\ref{interp}.
By applying Theorem~\ref{average-th} and
two lemmas in Section~\ref{SecondIn},
we will finally prove the second main result in this article.

\begin{theorem}\relabel{compare-K}
For any non-complete graph $G$ of order $n$,
$\epsilon(G,x)<\epsilon(K_n,x)$ holds for
all $x<0$.
\end{theorem}

\section{Proof of Theorem~\ref{compare-Q}
\relabel{firstIn}}

A vertex $u$ in a graph $G$ is called a {\it simplicial
vertex} if
$\{u\}\cup N_G(u)$ is a clique of $G$,
where $N_G(u)$ is the set of vertices in
$G$ which are adjacent to $u$.
For a simplicial vertex $u$ of $G$,
$P(G,x)$ has the following property
(see~\cite{DKT2005, Read1968, RT1988}):
\begin{align}\relabel{sim-ch}
P(G,x)=(x-d(u))P(G-u,x),
\end{align}
where $G-u$ is the subgraph of $G$ induced by $V-\{u\}$
and $d(u)$ is the degree of $u$ in $G$.
By (\ref{sim-ch}), it is not difficult to show the following.

\begin{prop}\relabel{sim-epsi}
If $u$ is a simplicial vertex of a graph $G$, then
\begin{align}\relabel{sim-ch1}
\epsilon(G,x)=\frac{1}{x-d(u)}+\epsilon(G-u,x).
\end{align}
\end{prop}

It has been shown that a graph $Q$ of order $n$
is chordal if and only if
$Q$ has an ordering $u_1,u_2,\ldots,u_n$
of its vertices such that
$u_i$ is a simplicial vertex in
$Q[\{u_1,u_2,\ldots,u_i\}]$
for all $i=1,2,\ldots,n$
(see \cite{Dirac1961, FG1965}).
Such an ordering of vertices in $Q$ is
called a {\it perfect elimination ordering} of $Q$.
For any perfect
elimination ordering $u_1,u_2,\ldots,u_n$ of
a chordal graph $Q$,
by Proposition~\ref{sim-epsi},
\begin{align}\relabel{sim-ch2-1}
\epsilon(Q,x)=\sum_{i=1}^n \frac 1{x-d_{Q_i}(u_i)},
\end{align}
where $Q_i$ is the subgraph $Q[\{u_1,u_2,\ldots,u_i\}]$.

Now we are ready to prove Theorem~\ref{compare-Q}.

\vspace{0.2 cm}

\noindent {\it Proof of Theorem~\ref{compare-Q}}:
Let $G$ be any graph of order $n$
and $Q$ be any chordal
and proper spanning subgraph of $G$.
When $n\le 3$, it is not difficult to verify that
$\epsilon(G,x)>\epsilon(Q,x)$ holds for all $x<0$.

Suppose that Theorem~\ref{compare-Q} fails
and $G=(V,E)$ is a counter-example to this result
such that $|V|+|E|$ has the minimum value among all
counter-examples. Thus the result holds
for any graph $H$ with $|V(H)|+|E(H)|<|V|+|E|$
and any chordal and proper spanning subgraph $Q'$ of $H$,
but $G$ has a chordal and proper spanning subgraph $Q$ such that $\epsilon(G,x)\le \epsilon(Q,x)$ holds for some $x<0$.

We will establish the following claims.
Let $u_1,u_2,\ldots,u_n$ be a perfect elimination ordering
of $Q$ and $Q_i=Q[\{u_1,\ldots,u_i\}]$ for all
$i=1,2,\ldots,n$.
So $u_i$ is a simplicial vertex of $Q_i$ for
all $i=1,2,\ldots,n$.

\noindent {\bf Claim 1}:
$u_n$ is not a simplicial vertex of $G$.

Note that $Q-u_n$ is chordal and a spanning
subgraph of $G-u_n$.
By the assumption on the minimality of $|V|+|E|$,
$\epsilon(G-u_n,x)\ge \epsilon(Q-u_n,x)$ holds
for all $x<0$,
where the inequality is strict whenever
$Q-u_n\not\cong G-u_n$.

Clearly $d_G(u_n)\ge d_Q(u_n)$.
As $Q$ is a proper subgraph of $G$,
$d_G(u_n)>d_Q(u_n)$
in the case that $G-u_n\cong Q-u_n$.
If $u_n$ is also a simplicial vertex of $G$,
then by Proposition~\ref{sim-epsi},
\begin{align}\relabel{comp-GQ}
\epsilon(G,x)=\frac 1{x-d_G(u_n)}+\epsilon(G-u_n,x),
\quad
\epsilon(Q,x)=\frac 1{x-d_Q(u_n)}+\epsilon(Q-u_n,x),
\end{align}
implying that $\epsilon(G,x)>\epsilon(Q,x)$ holds for all
$x<0$, a contradiction.
Hence Claim 1 holds.

\noindent {\bf Claim 2}: $d_G(u_n)>d_Q(u_n)$.

Clearly $d_G(u_n)\ge d_Q(u_n)$.
Since $u_n$ is a simplicial vertex of $Q$
and $Q$ is a subgraph of $G$,
$d_G(u_n)=d_Q(u_n)$ implies that $u_n$ is a simplicial
vertex of $G$,
contradicting Claim 1.
Thus Claim 2 holds.

For any edge $e$ in $G$, let $G-e$ be the graph
obtained from $G$ by deleting $e$. Let $G/e$ be
the graph obtained from $G$ by contracting $e$
and replacing multiple edges, if any arise,
by single edges.

\noindent {\bf Claim 3}:
For any $e=u_nv\in E-E(Q)$,
both $\epsilon(G-e,x)\ge \epsilon(Q,x)$
and $\epsilon(G/e,x)\ge \epsilon(Q-u_n,x)$
hold for all $x<0$.

As $e=u_nv\in E-E(Q)$,
$Q$ is a spanning subgraph of $G-e$ and $Q-u_n$ is a spanning subgraph of $G/e$.
As both $Q$ and $Q-u_n$ are chordal, by the assumption on
the minimality of $|V|+|E|$, the theorem holds for both $G-e$ and $G/e$.
Thus this claim holds.

\noindent {\bf Claim 4}:
$\epsilon(G,x)>\epsilon(Q,x)$ holds for all
$x<0$.

By Claim 2, there exists $e=u_nv\in E-E(Q)$.
By Claim 3, $\epsilon(G-e,x)\ge \epsilon(Q,x)$
and $\epsilon(G/e,x)\ge \epsilon(Q-u_n,x)$
hold for all $x<0$.
By (\ref{epsi-G}) and (\ref{sim-ch2-1}),
\begin{eqnarray}\relabel{G-1-eq1}
\  &   & (\epsilon(G-e,x)-\epsilon(Q,x))
\times (-1)^nP(G-e,x)\nonumber \\
& = & (-1)^nP'(G-e,x)+(-1)^{n+1}P(G-e,x)\sum_{i=1}^n
\frac 1{x-d_{Q_i}(u_i)}.
\end{eqnarray}
As $(-1)^nP(G-e,x)>0$ and
$\epsilon(G-e,x)\ge \epsilon(Q,x)$ for all $x<0$,
the left-hand side of (\ref{G-1-eq1}) is non-negative for $x<0$,
implying that
the right-hand side of (\ref{G-1-eq1}) is also non-negative
for $x<0$,
i.e., 
\begin{eqnarray}\relabel{G-1-eq1-1}
(-1)^nP'(G-e,x)+(-1)^{n+1}P(G-e,x)\sum_{i=1}^n
\frac 1{x-d_{Q_i}(u_i)}\ge 0,\quad \forall x<0.
\end{eqnarray}
As $u_1,\ldots,u_{n-1}$
is a perfect elimination ordering of $Q-u_n$
and  $\epsilon(G/e,x)\ge \epsilon(Q-u_n,x)$ holds
for all $x<0$, similarly we have:
\begin{align}\relabel{G-1-eq1-2}
(-1)^{n-1}P'(G/e,x)+(-1)^{n}P(G/e,x)\sum_{i=1}^{n-1}
\frac 1{x-d_{Q_i}(u_i)}\ge 0,\quad \forall x<0.
\end{align}
As $(-1)^{n-1}P(G/e,x)>0$ holds for all $x<0$,
(\ref{G-1-eq1-2}) implies that
\begin{eqnarray}\relabel{G-1-eq1-3}
\  &   &
(-1)^{n-1}P'(G/e,x)+(-1)^{n}P(G/e,x)\sum_{i=1}^{n}
\frac 1{x-d_{Q_i}(u_i)}
\nonumber \\
& \ge &
\frac {(-1)^{n}P(G/e,x)}{x-d_{Q_n}(u_n)} >0,
\qquad \forall x<0.
\end{eqnarray}
By the deletion-contraction formula for chromatic
polynomials,
\begin{align}\label{del-con}
P(G, x)=P(G-e, x)-P(G/e, x),\quad
P'(G, x)=P'(G-e, x)-P'(G/e, x).
\end{align}
Then (\ref{G-1-eq1-1}), (\ref{G-1-eq1-3}) 
and (\ref{del-con}) imply that
\begin{align}\relabel{G-1-eq1-4}
(-1)^nP'(G,x)+(-1)^{n+1}P(G,x)\sum_{i=1}^n
\frac 1{x-d_{Q_i}(u_i)}> 0,\quad \forall x<0.
\end{align}
By (\ref{epsi-G}) and (\ref{sim-ch2-1}),
inequality (\ref{G-1-eq1-4}) implies that 
\begin{align}\relabel{G-1-eq1-5}
\left (\epsilon(G,x)-\epsilon(Q,x)\right )(-1)^nP(G,x)
> 0,\quad \forall x<0.
\end{align}
Since $(-1)^nP(G,x)>0$ holds for all $x<0$,
inequality (\ref{G-1-eq1-5}) implies 
Claim 4.

As Claim 4 contradicts the assumption of $G$,
there are no counter-examples to this result and
the theorem is proved.
\qed

\section{An approach for proving
Theorem~\ref{compare-K} \relabel{SecondIn}}

In this section, we will mainly show that, in order
to prove Theorem~\ref{compare-K},
it suffices to prove Theorem \ref{average-th}.
By (\ref{sim-ch2-1}),
we have
\begin{align}\relabel{Kn-epsi}
\epsilon(K_n,x)=\sum\limits_{i=0}^{n-1}\frac{1}{x-i}.
\end{align}
Thus,
\begin{align}\relabel{Kn-epsi2}
\epsilon(K_n,x)-\epsilon(G,x)
=\frac{(-1)^n}{P(G,x)}
\left (
(-1)^{n}P(G, x)\sum_{i=0}^{n-1}\frac {1}{x-i}
+(-1)^{n+1}P'(G, x)\right ).
\end{align}

For any graph $G$ of order $n$, define
\begin{align}
\xi(G, x)=(-1)^{n}P(G, x)\sum_{i=0}^{n-1}\frac {1}{x-i}+(-1)^{n+1}P'(G, x).
\relabel{xi}
\end{align}
Note that $\xi(G,x)\equiv 0$ if $G$ is a complete graph.
For any non-complete graph $G$ and any $x<0$,
we have $(-1)^nP(G,x)>0$ and so (\ref{Kn-epsi2})
implies that
$\epsilon(K_n,x)-\epsilon(G,x)>0$
if and only if $\xi(G,x)> 0$.

\begin{prop}\relabel{compare-K-eq}
Theorem~\ref{compare-K} holds
if and only if
$\xi(G,x)> 0$ holds for every non-complete graph $G$
and all $x<0$.
\end{prop}

It can be easily verified that $\xi(G,x)>0$ holds
for all non-complete graphs $G$ of order at most $3$
and all $x<0$.
For the general case, we will prove it by induction.
In the rest of this section,
we will find a relation between $\xi(G,x)$ and $\xi(G-u,x)$
for a vertex $u$ in $G$ in two cases.
Lemma~\ref{ud0} is for the case when $u$ is a simplicial vertex and 
Lemma~\ref{rec2} when $d(u) \ge 1$.
We then explain why
Theorem~\ref{average-th} implies 
$\xi(G,x)>0$ for all non-complete graphs $G$
and all $x<0$.

\begin{lemma}\relabel{ud0}
Let $G$ be a graph of order $n$.
If $u$ is a simplicial vertex of $G$ with $d(u)=d$, then
\begin{align}\relabel{ud0-eq1}
\xi(G, x)=(d-x)\xi(G-u, x)
+\frac{(-1)^{n-1}(n-1-d)P(G-u,x)}{n-1-x}.
\end{align} 
\end{lemma}

\begin{proof}
As $u$ is a simplicial vertex of $G$ with $d(u)=d$,
$P(G,x)=(x-d)P(G-u,x)$ by (\ref{sim-ch}).
Thus
$
P'(G, x)=P(G-u, x)+(x-d)P'(G-u, x).
$
By (\ref{xi}),
\begin{eqnarray}
\xi(G,x)
&=&(-1)^n (x-d)P(G-u,x)\sum_{i=0}^{n-1}\frac 1{x-i}
+(-1)^{n+1}(P(G-u, x)+(x-d)P'(G-u, x))\nonumber \\
&=&(d-x)\xi(G-u,x)+\frac{(-1)^n(x-d)P(G-u,x)}{x-n+1}
+(-1)^{n+1}P(G-u, x)\nonumber \\
&=&(d-x)\xi(G-u, x)+
\frac{(-1)^{n-1}(n-1-d)P(G-u,x)}{n-1-x}.
\end{eqnarray}
\end{proof}

Note that
$d\le n-1$ and $(-1)^{n-1} P(G-u,x)>0$ holds for all $x<0$,
implying that the second term in the right-hand side of
(\ref{ud0-eq1}) is non-negative.
Thus, if $u$ is a simplicial vertex of $G$ and $x<0$,
by Lemma~\ref{ud0},
$\xi(G-u,x)>0$ implies that $\xi(G,x)>0$.

Now consider the case that
$u$ is a vertex in $G$ with $d(u)=d\ge 1$.
Assume that $N(u)=\{u_1,u_2,\ldots,u_d\}$.
For any $i=1, 2, \ldots, d-1$,
let $G_i$ denote the graph obtained from $G-u$ by adding
edges joining $u_i$ to $u_j$ whenever $u_iu_j\notin E(G)$
for all $j$ with $i+1\le j\le d$.
Thus, $u_i$ is adjacent to $u_j$ in $G_i$ for all $j$ 
with $i+1\le j\le d$.
In the case that $u$ is a simplicial vertex of $G$, 
$G_i\cong G-u$ for all $i=1,2,\cdots,d-1$.
By applying the deletion-contraction formula
for chromatic polynomials
(see \cite{DKT2005,Read1968}),
$P(G,x)$ can be expressed in terms of $P(G-u,x)$ and $P(G_i,x)$ for
$i=1,2,\cdots,d-1$.

\begin{lemma}\relabel{rec0}
Let $u$ be a vertex in $G$ with $d(x)=d\ge 1$
and for $i=1,2,\cdots,d-1$, let $G_i$ be the graph defined above. Then,
\begin{align}
P(G, x)=(x-1)P(G-u, x)-\sum_{i=1}^{d-1}P(G_i, x).
\relabel{rec1}
\end{align}
\end{lemma}

\begin{proof}
For $1\le i\le d$,
let $E_i$ denote the set of edges $uu_j$ in $G$ for $j=1,2,\cdots,i-1$.
So $|E_i|=i-1$ and $E_1=\emptyset$.
For any $i$ with $1\le i\le d-1$,
applying the deletion-contraction formula
for chromatic polynomials to edge $uu_i$ in $G-E_i$,
the graph obtained from $G$ by removing
all edges in $E_i$,
we have
\begin{align}
P(G-E_i, x)=P(G-E_{i+1}, x)-P((G-E_i)\slash uu_i, x)
=P(G-E_{i+1}, x)-P(G_i, x),
\relabel{rec0-1}
\end{align}
where the last equality follows from the fact that
$(G-E_i)\slash uu_i\cong G_i$
by the assumption of $G_i$.
Thus, by (\ref{rec0-1}),
\begin{align}
P(G,x)=P(G-E_1,x)=P(G-E_d,x)-\sum_{i=1}^{d-1}P(G_i,x).
\relabel{rec0-2}
\end{align}
As $u$ is of degree $1$ in $G-E_d$,
$P(G-E_d,x)=(x-1)P(G-u,x)$.
Hence  (\ref{rec1}) follows.
\end{proof}

\begin{lemma}\relabel{rec2}
Let $G$ be a graph of order $n$ and let $u$ be
a vertex of $G$ with $d(u)=d\ge 1$. Then
\begin{align}\relabel{rec2-eq1}
\xi(G, x)=(1-x)\xi(G-u, x)+\sum_{i=1}^{d-1}\xi(G_i, x)
+\frac{(-1)^{n}\left[(x-n+1)P(G-u, x)-P(G, x)\right]}{n-x-1},
\end{align}
where $G_1,\ldots,G_{d-1}$ are graphs defined above.
\end{lemma}

\begin{proof}
By (\ref{rec1}), we have
\begin{align}\relabel{rec2-eq0}
P'(G, x)=P(G-u, x)+(x-1)P'(G-u, x)-\sum_{i=1}^{d-1}P'(G_i, x).
\end{align}
Thus
\begin{eqnarray}
\ \xi(G, x) & = & (-1)^{n}P(G, x)
\sum_{j=0}^{n-1}\frac {1}{x-j}+(-1)^{n+1}P'(G, x) \nonumber \\
& = & (-1)^{n}\left[(x-1)P(G-u, x)-\sum_{i=1}^{d-1}P(G_i, x)\right]\sum_{j=0}^{n-1}\frac {1}{x-j} \nonumber \\
&   & +(-1)^{n+1}\left[P(G-u, x)+(x-1)P'(G-u, x)-\sum_{i=1}^{d-1}P'(G_i, x)\right] \nonumber \\
& = & (1-x)\left[(-1)^{n-1}P(G-u, x)
\sum_{j=0}^{n-2}\frac {1}{x-j}
+(-1)^{n}P'(G-u, x)\right]  \nonumber \\
&   & +\sum_{i=1}^{d-1}\left[(-1)^{n-1}P(G_i, x)
\sum_{j=0}^{n-2}\frac {1}{x-j}
+(-1)^{n}P'(G_i, x)\right] +(-1)^{n+1}P(G-u, x)
\nonumber \\
&   & +(-1)^{n}\left[\frac{(x-1)P(G-u,x)}{x-(n-1)}
      -\frac{1}{x-(n-1)}\sum_{i=1}^{d-1} P(G_i,x)\right]
\nonumber 
\end{eqnarray}

\begin{eqnarray}
& = & (1-x)\xi(G-u, x)+\sum_{i=1}^{d-1}\xi(G_i, x)
\nonumber \\ & &
+\frac{(-1)^{n}\left[(x-n+1)P(G-u, x)
-P(G, x)\right]}{n-x-1},
\relabel{rec2-eq2}
\end{eqnarray}
where the last expression follows from (\ref{rec1})
and the definitions
of $\xi(G-u, x)$ and $\xi(G_i,x)$.
The result then follows.
\end{proof}

It is known that $\xi(G,x)>0$ holds for all
non-complete graphs $G$ of order at most $3$ and all $x<0$.
For any non-complete graph $G$ of order $n\ge 4$,
by Lemma~\ref{ud0}, 
$\xi(G-u,x)>0$ implies $\xi(G,x)>0$ 
for each simplicial vertex $u$ in $G$ 
and all $x<0$;
by Lemma~\ref{rec2},
for any $x<0$,
$\xi(G-u,x)>0$ implies $\xi(G,x)>0$
whenever $u$ is an non-isolated vertex in $G$ 
satisfying the following 
inequality:
\begin{align} \relabel{ineq2}
(-1)^{n}((x-n+1)P(G-u, x)-P(G, x))> 0.
\end{align}  
Note that the left-hand side of (\ref{ineq2})
vanishes when $G$ is $K_n$.
Also notice that there exist
non-complete graph $G$ and some vertex $u$ in $G$ 
such that inequality (\ref{ineq2}) does not hold
for some $x<0$.
For example, if $G$ is the complete bipartite graph $K_{2,3}$
and $u$ is a vertex of degree $3$ in $G$, then (\ref{ineq2}) fails for all real $x$ with $-2.3<x<0$.
However,
to prove that for any $x<0$, there exists some vertex $u$ in $G$
such that inequality (\ref{ineq2}) holds,
it suffices to prove the following inequality 
(i.e., Theorem~\ref{average-th}):
\begin{align}
(-1)^n (x-n+1)\sum_{u\in V}P(G-u, x)
+(-1)^{n+1}nP(G, x)
> 0 \relabel{average}
\end{align}
for any non-complete graph $G=(V,E)$ of order $n$ and all $x<0$.

By Proposition~\ref{compare-K-eq} and
inequality (\ref{ineq2}),
to prove Theorem~\ref{compare-K},
we can now just focus on
proving inequality~(\ref{average})
(i.e., Theorem~\ref{average-th}).
The proof of Theorem~\ref{average-th}
will be given in Section~\ref{finalproof}
based on the interpretations for the coefficients
of chromatic polynomials introduced in
Section~\ref{interp}.

\section{Combinatorial interpretations for
coefficients of $P(G,x)$\relabel{interp}}

Let $G=(V,E)$ be any graph.
In this section, we will introduce
Greene $\&$ Zaslavsky's
combinatorial interpretation in \cite{GZ1983}  
for the coefficients of
$P(G,x)$ in terms of acyclic orientations.
The result will be applied in the next section to prove 
Theorem~\ref{average-th}.

An orientation $D$ of $G$ is called {\it acyclic}
if $D$ does not contain any directed cycle.
Let $\alpha (G)$ be the number of acyclic
orientations of a graph $G$.
In~\cite{Stanley1973}, Stanley
gave a nice combinatorial interpretation
of $(-1)^n P(G, -k)$ for any positive integer $k$
in terms of acyclic
orientations of $G$. In particular, he proved:

\begin{theorem}
[\cite{Stanley1973}]\relabel{Stanley}
For any graph $G$ of order $n$,
$(-1)^nP(G, -1)=\alpha(G)$,
i.e.,
\begin{align}\label{Stanley-eq1}
\sum\limits_{i=1}^{n}a_i(G)=\alpha(G).
\end{align}
\end{theorem}

In a digraph $D$, any vertex of $D$
with in-degree (resp. out-degree) zero is called a
{\it source} (resp. {\it sink}) of $D$.
It is well known that
any acyclic digraph has at least one
source and at least one sink.
If $v$ is an isolated vertex of $G$,
then $v$ is a source and also a sink in
any orientation of $G$.

For any $v\in V$,
let $\alpha(G, v)$ be the number of acyclic
orientations of $G$ with $v$ as its
unique source.
Clearly $\alpha(G, v)=0$ if and only if $G$ is
not connected.
In 1983, Greene and Zaslavsky \cite{GZ1983}
showed that $a_1(G)=\alpha(G, v)$.

\begin{theorem}[\cite{GZ1983}]\relabel{source}
For any graph $G=(V,E)$,
$a_1(G)=\alpha(G,v)$ holds for every $v\in V$.
\end{theorem}

This theorem was proved originally by using the
theory of hyperplane arrangements. See \cite{GS2000}
for three other nice proofs.

By Whitney's Broken-cycle Theorem
(i.e., Theorem~\ref{brokencycle}), $a_i(G)$
equals the number of spanning subgraphs of $G$ with
$i$ components and $n-i$ edges, containing no broken
cycles of $G$. In particular, $a_1(G)$ is the
number of spanning trees of $G$ containing no broken
cycles of $G$. Now we have two different combinatorial
interpretations for $a_1$.
For any $a_i(G)$, $2\leq i\leq n$,
its combinatorial interpretation can be obtained by 
applying these two different combinatorial
interpretations for $a_1$.

Let $\mathcal{P}_i(V)$ be the set of
partitions $\{V_1,V_2,\ldots,V_i\}$ of $V$
such that $G[V_j]$ is connected for all $j=1,2,\ldots,i$
and let $\beta_i(G)$ be the number
of ordered pairs $(P_i, F)$, where
\begin{enumerate}
\item[(a)] $P_i=\{V_1,V_2,\ldots,V_i\}\in \mathcal{P}_i(V)$;
\item[(b)]
$F$ is a spanning forest of $G$ with exactly $i$ components
$T_1,T_2, \ldots, T_i$, where each $T_j$ is a spanning tree of
$G[V_j]$ containing no broken cycles of $G$.
\end{enumerate}

For any subgraph $H$ of $G$,
let $\widetilde{\tau}(H)$ be the number
of spanning trees of $H$ containing
no broken cycles of $G$.
By Theorem~\ref{brokencycle},
$\widetilde{\tau}(H)=a_1(H)$ holds
and the next result follows.

\begin{theorem}\relabel{interpre1}
For any graph $G$ and any $1\leq i\leq n$,
\begin{align}
a_i(G)=\beta_i(G)=\sum_{\{V_1,\ldots, V_i\}\in \mathcal{P}_i(V)}\prod_{j=1}^{i} \widetilde{\tau}(G[V_j]).
\end{align}
\end{theorem}

Now let $V=\{1, 2, \ldots, n\}$.
For any $i:1\le i\le n$
and any vertex $v\in V$, let
$\mathcal{OP}_{i, v}(V)$ be the family of
 ordered partitions $(V_1,V_2,\ldots,V_i)$ of $V$
 such that

\begin{enumerate}
\item[(a)] $\{V_1,V_2,\ldots,V_i\}\in \mathcal{P}_i(V)$,
where $v\in V_1$;
\item[(b)] for $j=2, \ldots, i$, the minimum number
in the set $\bigcup_{j\le s\le i} V_s$ is within $V_j$.
\end{enumerate}
Clearly, for any $v\in V$ and
any $\{V_1,V_2,\ldots,V_i\}\in \mathcal{P}_i(V)$,
there is exactly one permutation
$(\pi_1,\pi_2,\ldots,\pi_i)$ of $1,2,\ldots,i$
such that $(V_{\pi_1}, V_{\pi_2},\ldots,V_{\pi_i})\in
\mathcal{OP}_{i, v}(V)$.

By Theorem \ref{source},
$\widetilde{\tau}(G[V_j])=\alpha(G[V_j],u)$ holds
for any vertex $u$ in $G[V_j]$ and
Theorem \ref{interpre1} is equivalent to 
a result in \cite{GZ1983} 
which we illustrate differently below.

\begin{theorem}[\cite{GZ1983}, Theorem 7.4]\relabel{interpre2}
For any $v\in V$ and any $1\leq i\leq n$,
\begin{align}
a_i(G)=\sum_{(V_1,\ldots, V_i)\in \mathcal{OP}_{i, v}(V)}\alpha(G[V_1],v) \prod_{j=2}^{i} \alpha(G[V_j],m_j),
\relabel{interpre}
\end{align}
where $m_j$ is the minimum number in $V_j$ for
$j=2,\ldots,i$.
\end{theorem}

Note that the theorem above indicates that the
right hand side of (\ref{interpre}) is independent
of the choice of $v$. Thus, for any $1\le i\le n$,
\begin{align}
na_i(G)=\sum_{v\in V}
\sum_{(V_1,\ldots, V_i)\in \mathcal{OP}_{i, v}(V)}
\alpha(G[V_1],v) \prod_{j=2}^{i} \alpha(G[V_j],m_j).
\relabel{interpre-n}
\end{align}

Let $P^{(i)}(G, x)$ be the $i$-th derivative of $P(G,x)$.  
Very recently, 
Bernardi and Nadeau \cite{Bernardi2020} gave an interpretation of $P^{(i)}(G, -j)$ 
for any nonnegative integers $i$ and $j$
in terms of acyclic orientations.
When $i=0$,  their result is exactly
Theorem \ref{Stanley} due to Stanley~\cite{Stanley1973};
and when $j=0$, it is Theorem \ref{interpre2} due to 
Greene $\&$ Zaslavsky~\cite{GZ1983}.

\section{Proofs of Theorems \ref{average-th} and \ref{compare-K}
\relabel{finalproof}}

By the explanation in Section 3, 
to prove Theorem~\ref{compare-K}, 
it suffices to prove Theorem~\ref{average-th}.
In this section, we will
prove Theorem~\ref{average-th} by
showing that
the coefficient of
$x^i$ in the expansion of the left-hand side of (\ref{right-2})
in Theorem~\ref{average-th}
is of the form $(-1)^i d_i$
with $d_i\ge 0$
for all $i=1,2,\ldots,n$.
Furthermore,
$d_i>0$ holds for some $i$ when $G$ is not complete.

We first establish the following result.

\begin{lemma}\relabel{le5-1}
Let $G=(V,E)$ be a non-complete graph of order
$n\geq 3$
and component number $c$.
\begin{enumerate}
\renewcommand{\theenumi}{\rm (\alph{enumi})}
\item If $c=1$ and $G$ is not the $n$-cycle $C_n$,
then there exist non-adjacent vertices $u_1,u_2$
of $G$ such that $G-\{u_1,u_2\}$ is connected.
\item If $2\le c\le n-1$,
then for any integer $i$ with $c\le i\le n-1$,
there exists a partition $V_1,V_2,\ldots,V_i$ of $V$
such that $G[V_j]$ is connected for all $j=2,\ldots,i$
and $G[V_1]$ has exactly two components
one of which is an isolated vertex.
\end{enumerate}
\end{lemma}

\begin{proof} (a). As $c=1$, $G$ is connected.
As $G$ is non-complete, the result is trivial when
$G$ is 3-connected.

If $G$ is not $2$-connected, choose vertices
$u_1$ and $u_2$ from distinct blocks $B_1$ and $B_2$ of $G$
such that both $u_1$ and $u_2$ are not cut-vertices of $G$.
Then $u_1u_2\notin E(G)$ and $G-\{u_1,u_2\}$ is connected.

Now consider the case that $G$ is 2-connected
but not $3$-connected.
Since $G$ is not $C_n$, there exists a vertex $w$ such that
$d(w)\geq 3$.
If $d(w)=n-1$, then
$G-\{u_1,u_2\}$ is connected for any two non-adjacent
vertices $u_1$ and $u_2$ in $G$.
If $G-w$ is $2$-connected and $d(w)\leq n-2$,
then $G-\{w,u\}$ is connected
for any $u\in V-N_G(w)$.
If $G-w$ is not $2$-connected,
then $G-w$ contains two non-adjacent vertices $u_1,u_2$
such that $G-\{w,u_1,u_2\}$ is connected,
implying that $G-\{u_1,u_2\}$ is connected
as  $d(w)\geq 3$.

(b).
Let $G_1, G_2,\ldots, G_c$ be the components of $G$ with
$|V(G_1)|\ge |V(G_j)|$ for all $j=1,2,\ldots,c$.
As $c\le n-1$, $|V(G_1)|\ge 2$.
Choose $u\in V(G_1)$ such that $G_1-u$ is connected. Then
$V(G_2)\cup \{u\}, V(G_1)-\{u\}, V(G_3),\ldots, V(G_c)$
is a partition of $V$ satisfying the condition in (b) for $i=c$.

Assume that (b) holds for $i=k$, where $c\le k<n-1$, 
and $V_1,V_2,\ldots,V_k$ is a partition of $V$ 
satisfying the condition in (a).
Then $G[V_1]$ has an isolated vertex $u$
and $G[V'_1]$ is connected, where $V'_1=V_1-\{u\}$.
Since $k\le n-2$, either $|V'_1|\ge 2$
or $|V_j|\ge 2$ for some $j\ge 2$.

If $|V'_1|\ge 2$,  then $V'_1$ has a partition $V'_{1,1}, V'_{1,2}$
such that both $G[V'_{1,1}]$ and $G[V'_{1,2}]$ are connected,
implying that
$V'_{1,1}\cup \{u\}, V'_{1,2}, V_2, V_3,\ldots, V_k$
is a partition of $V$ satisfying the condition in (b)
for $i=k+1$.

Similarly,
if $|V_j|\ge 2$ for some $j\ge 2$ (say $j=2$),
then $V_2$ has a partition $V_{2,1}, V_{2,2}$
such that both $G[V_{2,1}]$ and $G[V_{2,2}]$ are connected,
implying that $V_{1}, V_{2,1},V_{2,2}, V_3,\ldots, V_k$
is a partition of $V$ satisfying the condition in (b)
for $i=k+1$.
\end{proof}

For any graph $G=(V,E)$ of order $n$, write
\begin{align}\relabel{new-coe}
(-1)^n \left[(x-n+1)\sum_{u\in V(G)}P(G-u, x)-nP(G, x)\right]=\sum_{i=1}^{n}(-1)^i d_ix^i.
\end{align}
By comparing coefficients, it can be shown that
\begin{align}\relabel{ci-exp}
d_i=\sum_{u\in V(G)}\left
[a_{i-1}(G-u)+(n-1)a_i(G-u)\right ]-na_{i}(G),
\quad \forall i=1, 2, \ldots, n.
\end{align}
It is obvious that when $G$ is the
complete graph $K_n$, the left-hand side of
(\ref{new-coe}) vanishes and thus
$d_i=0$ for all $i=1,2,\ldots,n$.
Now we consider the case that $G$ is not complete.

\begin{prop}\relabel{pos-d}
Let $G=(V,E)$ be a non-complete graph of order $n$
and component number $c$.
Then, for any $i=1,2,\ldots,n$, $d_i\ge 0$ 
and equality holds
if and only if one of the following cases happens:
\begin{enumerate}
\renewcommand{\theenumi}{\rm (\alph{enumi})}
\item $i=n$;
\item $1\le i\le c-2$;
\item $i=c-1$ and $G$ does not have isolated vertices;
\item $i=c=1$ and $G$ is  $C_n$.
\end{enumerate}
\end{prop}

\begin{proof}
We first
show that $d_i=0$ in any one of the four cases above.

By (\ref{ci-exp}),
$d_n=\sum_{u\in V}\left[1+(n-1)\cdot 0\right]-n\cdot1=0$.

It is known that for $1\le i\le n$,
$a_i(G)=0$ if and only if $i<c$ (see~\cite{DKT2005,Read1968,RT1988}).
Similarly, $a_i(G-u)=0$ for all $i$ with $1\le i<c-1$
and all $u\in V$,
and $a_{c-1}(G-u)=0$ if $u$ is not an isolated
vertex of $G$.
By (\ref{ci-exp}), $d_i=0$ for all $i$ with
$1\le i\leq c-2$,
and $d_{c-1}=0$ when $G$ does not have isolated vertices.

If $G$ is $C_n$,
then $a_1(G)=n-1$, $a_0(G-u)=0$ and $a_1(G-u)=1$ for each
$u\in V$, implying that $d_1=0$ by (\ref{ci-exp}).

In the following, we will show that
$d_i>0$ when $i$ does not belong to any one
of the four cases.

If $G$ has isolated vertices, then
$a_{c-1}(G-u)>0$ for any isolated vertex $u$ of $G$ and
\begin{align}\label{isolated}
\sum_{u\in V}a_{c-1}(G-u)=
\sum_{u\in V\atop
u \text{ isolated}}a_{c-1}(G-u)>0.
\end{align}
As $a_{c-1}(G)=0$, by (\ref{ci-exp}), we have $d_{c-1}>0$
in this case.
Now it remains to show that
$d_i>0$ holds for all $i$ with $c\le i\le n-1$,
except when $i=c=1$ and $G$ is $C_n$.

For any $v\in V$, let
$\mathcal{OP}'_{i,v}(V)$ be the set of
ordered partitions $(V_1,\ldots,V_i)\in \mathcal{OP}_{i,v}(V)$
with $V_1=\{v\}$.
As $\alpha(G[V_1],v)=1$, for any $i$ with
$c\leq i\leq n$,
by Theorem~\ref{interpre2},
\begin{align}\relabel{G-u-i-1}
a_{i-1}(G-v)
=\sum_{(V_1,\ldots, V_i)\in \mathcal{OP}'_{i,v}(V)}
\alpha(G[V_1],v) \prod_{j=2}^{i} \alpha(G[V_j],m_j),
\end{align}
where $m_j$ is the minimum number in $V_j$ for all
$j=2,\ldots,i$.

Let $s$ and $v$ be distinct members in $V$.
For any $V_1\subseteq V-\{s\}$ with $v\in V_1$,
let $\alpha(G[V_1\cup \{s\}],v,s)$ be the
number of those acyclic orientations of $G[V_1\cup \{s\}]$
with $v$ as the unique source and $s$ as one sink.
Then $\alpha(G[V_1\cup \{s\}],v,s)\le
\alpha(G[V_1],v)$ holds,
where the inequality is strict if and only if
$G[V_1]$ is connected but
$G[V_1\cup \{s\}]$ is not.
Observe that
\begin{align}
a_i(G-s)
&=\sum_{(V_1,\ldots,V_i)\in \mathcal{OP}_{i,v}(V-\{s\})}
\alpha(G[V_1],v) \prod_{j=2}^{i} \alpha(G[V_j],m_j)
\nonumber \\
&\ge \sum_{(V_1,\ldots,V_i)\in \mathcal{OP}_{i,v}(V-\{s\})}
\alpha(G[V_1\cup \{s\}],v,s) \prod_{j=2}^{i} \alpha(G[V_j],m_j)
\relabel{G-u-i-0} \\
&=\sum_{(V_1',\ldots,V_i')\in \mathcal{OP}_{i,v,s}(V)}
\alpha(G[V_1'],v,s) \prod_{j=2}^{i} \alpha(G[V_j'],m_j),
\relabel{G-u-i}
\end{align}
where $\mathcal{OP}_{i,v,s}(V)$ is the set of ordered partitions
$(V_1',\ldots,V_i')\in \mathcal{OP}_{i,v}(V)$ with
$s,v\in V_1'$.
By the explanation above,
inequality (\ref{G-u-i-0}) is strict
whenever $V-\{s\}$ has a partition $V_1,V_2,\ldots,V_i$
with $v\in V_1$
such that each $G[V_j]$ is connected
for all $j=1,2,\ldots,i$ but
$G[V_1\cup \{s\}]$ is not connected.

By (\ref{interpre-n}),  we have
\begin{eqnarray}
n a_i(G)
&=&\sum_{v\in V}
\sum_{(V_1,\ldots,V_i)\in \mathcal{OP}_{i,v}(V)}
\alpha(G[V_1],v) \prod_{j=2}^{i} \alpha(G[V_j],m_j)\nonumber \\
&=&\sum_{v\in V}
\sum_{(V_1,\ldots,V_i)\in \mathcal{OP}'_{i,v}(V)}
\alpha(G[V_1],v) \prod_{j=2}^{i} \alpha(G[V_j],m_j)
\nonumber \\
& &+\sum_{v\in V}
\sum_{(V_1,\ldots,V_i)\in
\mathcal{OP}_{i,v}(V)-\mathcal{OP}'_{i,v}(V)}
\alpha(G[V_1],v) \prod_{j=2}^{i} \alpha(G[V_j],m_j).
\relabel{ci-po-1}
\end{eqnarray}
By (\ref{G-u-i-1}),
\begin{eqnarray}
\sum_{v\in V}
\sum_{(V_1,\ldots,V_i)\in \mathcal{OP}'_{i,v}(V)}
\alpha(G[V_1],v) \prod_{j=2}^{i} \alpha(G[V_j],m_j)
=
\sum_{v\in V} a_{i-1}(G-v)
\relabel{ci-po-2},
\end{eqnarray}
and by (\ref{G-u-i}),
\begin{eqnarray}
& &\sum_{v\in V}
\sum_{(V_1,\ldots,V_i)\in \mathcal{OP}_{i,v}(V)-\mathcal{OP}'_{i,v}(V)}
\alpha(G[V_1],v) \prod_{j=2}^{i}
\alpha(G[V_j],m_j)\nonumber \\
&\le &
\sum_{v\in V}\sum_{s\in V-\{v\}}
\sum_{(V_1,\ldots,V_i)\in \mathcal{OP}_{i,v,s}(V)}
\alpha(G[V_1],v,s) \prod_{j=2}^{i} \alpha(G[V_j],m_j)
\relabel{ci-po-0} \\
&\leq &
\sum_{v\in V}\sum_{s\in V-\{v\}}
a_i(G-s)\relabel{ci-po-00} \\
&=&
(n-1)\sum_{v\in V}a_i(G-v),
\relabel{ci-po}
\end{eqnarray}
where inequality (\ref{ci-po-0}) is strict
if there exists $(V_1,\ldots,V_i)\in \mathcal{OP}_{i,v}(V)$
for some $v\in V$
such that $G[V_j]$ is connected for all $j=1,\ldots,i$
and $G[V_1]$ has acyclic orientations
with $v$ as the unique source
but with at least two sinks,
and by (\ref{G-u-i-0}) and (\ref{G-u-i}),
inequality (\ref{ci-po-00})
is strict if $V$ can be partitioned into $V_1,\ldots, V_i$
such that $G[V_j]$ is connected for all $j=2,\ldots,i$
but $G[V_1]$ has exactly two components,
one of which is an isolated vertex in $G[V_1]$.

As $G$ is not complete, by Lemma~\ref{le5-1}
and the above explanation,
the inequality of (\ref{ci-po})
is strict for all $i$ with $c\le i\le n-1$,
except when $i=c=1$ and $G$ is $C_n$.
Then, by (\ref{ci-po-1}), (\ref{ci-po-2}) and (\ref{ci-po}),
we conclude that
\begin{align}
d_i=\sum_{v\in V}\left [a_{i-1}(G-u)+(n-1)a_i(G-u)\right ]
-na_i(G)>0,\quad \forall c\le i\le n-1,
\end{align}
except that $i=c=1$ and $G$ is $C_n$.
Hence the proof is complete.
\end{proof}

Now everything is ready for proving Theorems~\ref{average-th} and \ref{compare-K}.

\medskip

\noindent
{\it Proof of Theorem~\ref{average-th}}:
Let $G$ be a non-complete graph of order $n$.
Recall (\ref{new-coe}) that
\begin{align}\label{proof-th3}
(-1)^n \left[(x-n+1)\sum_{u\in V(G)}P(G-u, x)-nP(G, x)\right]=\sum_{i=1}^{n}(-1)^i d_ix^i.
\end{align}
By Proposition \ref{pos-d}, we know that $d_i\geq 0$ for
all $i$ with $1\leq i\leq n$ and $d_{n-1}>0$.
Thus
$\sum_{i=1}^{n}(-1)^i d_ix^i>0$ holds
for all $x<0$, which completes the proof of
Theorem~\ref{average-th}.
\qed

\medskip

\begin{prop}\relabel{pro5-3}
For any non-complete graph $G$, $\xi(G,x)>0$ holds
for all $x<0$.
\end{prop}

\proof
We will prove this result by induction
on the order $n$ of $G$.
When $n=2$, the empty graph $N_2$ of order $2$
is the only non-complete graph of order $2$.
As $P(N_2,x)=x^2$, by (\ref{xi}),
we have
\begin{align}\label{proof-pro5}
\xi(N_2,x)=(-1)^2x^2\left ( \frac 1x +\frac 1{x-1}\right )
+(-1)^32x=\frac{x}{x-1}>0
\end{align}
for all $x<0$.

Assume that this result holds for any non-complete graph
$G$ of order less than $n$, where $n\ge 3$.
Now let $G$ be any non-complete graph of order $n$.

\noindent {\bf Case 1}:
$G$ contains an isolated vertex $u$.

By the inductive assumption,
$\xi(G-u,x)\ge 0$ holds for all $x<0$,
where equality holds when $G-u$ is a complete graph.
By Lemma~\ref{ud0},  $\xi(G,x)>0$ holds for all $x<0$.

\noindent {\bf Case 2}:
$G$ has no isolated vertex.

By Theorem~\ref{average-th},
(\ref{right-2}) holds for all $x<0$.
Thus, for any $x<0$, there exists some $u\in V(G)$ such that
$(-1)^n (x-n+1)P(G-u,x)+(-1)^{n+1}P(G,x)>0$
holds. Then, by Lemma~\ref{rec2} and by the inductive assumption,
 $\xi(G,x)>0$ holds for any $x<0$.

Hence the result holds.
\endproof

\noindent
{\it Proof of Theorem~\ref{compare-K}}:
It follows directly from
Propositions~\ref{compare-K-eq} and~\ref{pro5-3}.
\qed

\medskip

\section{Remarks and problems\relabel{further}}

First we give some remarks here.

\begin{enumerate}
\renewcommand{\theenumi}{\rm (\alph{enumi})}

\item Theorem~\ref{compare-K} implies that
for any non-complete graph $G$ of order $n$,
$\frac{P(G, x)}{P(K_n, x)}$
is strictly decreasing when $x<0$.

\item Let $G$ be a non-complete graph of order $n$
and $P(G, x)=\sum\limits_{i=1}^n (-1)^{n-i}a_i x^i$.
Then $\epsilon(G)<\epsilon(K_n)$ implies that
\begin{align}\label{average0}
\frac{a_1+2a_2+\cdots+na_n}{a_1+a_2+\cdots+a_n}> 1+\frac{1}{2}+\cdots+\frac{1}{n}.
\end{align}

\item When $x=-1$, Theorem~\ref{average-th}
implies that for any graph $G$ of order $n$,
\begin{align}
(-1)^{n-1}\sum_{u\in V}P(G-u, -1)\ge (-1)^nP(G, -1),
\relabel{average1}
\end{align}
where the inequality holds if and only if $G$
is complete.
By Stanley's interpretation for $(-1)^nP(G,-1)$
in~\cite{Stanley1973},
the inequality above implies that
for any graph $G=(V,E)$,
the number of acyclic
orientations of $G$ is at most
the total number of acyclic
orientations of $G-u$ for all $u\in V$,
where the equality holds if and only if $G$ is complete.
\end{enumerate}

Now we raise some problems for further study.

It is clear that for any graph $G$ of order $n$,
\begin{align}\label{con2-ex}
\frac{d}{dx}\left (\ln[(-1)^nP(G,x)]\right )
=\frac{P'(G,x)}{P(G,x)}<0
\end{align}
holds for all $x<0$.
We surmise that this property holds for higher derivatives of
the function $\ln[(-1)^nP(G,x)]$ in the interval $(-\infty,0)$.

\begin{conjecture}\relabel{con6-1}
Let $G$ be a graph of order $n$.
Then $\frac{d^k}{dx^k}\left (\ln[(-1)^nP(G,x)]\right )<0$
holds for all $k\ge 2$ and
$x\in (-\infty,0)$.
\end{conjecture}

Observe that $\epsilon(G,x)=\frac{d}{dx}\left (\ln[(-1)^nP(G,x)]\right )$.
We believe that
Theorems~\ref{compare-Q} and~\ref{compare-K}
can be extended to higher derivatives of
the function $\ln[(-1)^nP(G,x)]$.

\begin{conjecture}\relabel{con6-2}
Let $G$ be any non-complete graph
of order $n$
and $Q$ be any chordal and proper spanning subgraph $Q$ of $G$.
Then
\begin{align}\label{con3-ex}
\frac{d^k}{dx^k}\left (\ln[(-1)^nP(Q,x)]\right )
< \frac{d^k}{dx^k}\left (\ln[(-1)^nP(G,x)]\right )
<\frac{d^k}{dx^k}\left (\ln[(-1)^nP(K_n,x)]\right )
\end{align}
holds for any integer $k\geq 2$ and all $x<0$.
\end{conjecture}

It is not difficult to show that Conjecture~\ref{con6-1} holds
for $G\cong K_n$.
Thus the second inequality of Conjecture~\ref{con6-2}
implies Conjecture~\ref{con6-1}.

It is natural to extend the second part 
of Conjecture~\ref{mainconj} (i.e., $\epsilon(G)<\epsilon(K_n)$
for any non-complete graph $G$ of order $n$) 
to the inequality  
$\epsilon(G)\le \epsilon(G')$ for any graph $G'$ 
which contains $G$ as a subgraph.
However, this inequality is not always true.
Let $G_n$ denote the graph obtained from the complete bipartite graph 
$K_{2,n}$ by adding a new edge joining 
the two vertices in the partite set of size $2$.
Lundow and Markstr\"{o}m \cite{LM2006} 
stated that 
$\epsilon(K_{2,n})>\epsilon(G_n)$ holds for all $n\ge 3$.
In spite of this, 
 we believe that for any non-complete graph $G$,
we can add a new edge to $G$ to obtain a graph $G'$ 
with the property that 
$\epsilon(G)<\epsilon(G')$, as stated below.

\begin{conjecture}\relabel{con6-4}
For any non-complete graph $G$, there exist non-adjacent vertices $u$ and $v$ 
in $G$ such that $\epsilon(G)<\epsilon(G+uv)$.
\end{conjecture}

Obviously, Conjecture~\ref{con6-4} implies  
$\epsilon(G)<\epsilon(K_n)$ for any non-complete graph $G$ 
of order $n$ (i.e., Theorem~\ref{compare-K}).
Conjecture~\ref{con6-4} is similar to but may be not 
equivalent to the following conjecture
due to Lundow and Markstr\"{o}m \cite{LM2006}.

\begin{conjecture}[\cite{LM2006}]\relabel{con6-3}
For any $2$-connected graph $G$, 
there exists an edge $e$ in $G$ such that 
$\epsilon(G-e)<\epsilon(G)$.
\end{conjecture}

\section*{Acknowledgements}

The authors would like to thank the referees for their
helpful suggestions and comments.

\vspace{0.1 cm}

(F. Dong and E. Tay) Mathematics and Mathematics Education,
National Institute of Education,
Nanyang Technological University,
Singapore.
Email (Tay): engguan.tay@nie.edu.sg.

(J. Ge) School of Mathematical Sciences,
Sichuan Normal University, Chengdu, P. R. China.
Email: mathsgejun@163.com.

(H. Gong)
Department of Mathematics,
Shaoxing University, Shaoxing,
P. R. China.
Email: helingong@126.com.

(B. Ning)
College of Computer Science, Nankai University,
Tianjin 300071, P.R. China.
Email: ningbo-maths@163.com.

(Z. Ouyang)
Department of Mathematics,
Hunan First Normal University,
Changsha, P. R. China. Email:
oymath@163.com.

\end{document}